\def\bu{\bullet}
\def\marker{\>\hbox{${\vcenter{\vbox{
    \hrule height 0.4pt\hbox{\vrule width 0.4pt height 6pt
    \kern6pt\vrule width 0.4pt}\hrule height 0.4pt}}}$}\>}
\def\gpic#1{#1
     \smallskip\par\noindent{\centerline{\box\graph}} \medskip}
\documentclass[12pt]{article}
\usepackage[hmarginratio={1:1},vmarginratio={1:1},heightrounded,textwidth=420pt]{geometry}
\usepackage{amsmath,amssymb,amsfonts,amsthm,tikz,caption,subcaption,enumerate}
\usepackage{fullpage}

\usepackage[T1]{fontenc}

\usepackage{algorithm}
\usepackage[noend]{algpseudocode}
\usepackage{amsfonts}
\usepackage{hyperref}
\usepackage{tikz}
\usetikzlibrary{calc}
\usetikzlibrary{decorations.markings}
\usetikzlibrary{arrows}
\usetikzlibrary{patterns}
\usetikzlibrary{shapes,snakes}
\usepackage[textsize=small]{todonotes}
\usepackage{url}

\usepackage{todonotes}
\newcommand{\TODO}[1]{{\todo[inline]{#1}}}

\newtheorem{theorem}{Theorem}[section]

\newtheorem{corollary}[theorem]{Corollary}

\newtheorem{lemma}[theorem]{Lemma}
\newtheorem{proposition}[theorem]{Proposition}

\theoremstyle{definition}
\newtheorem{definition}[theorem]{Definition}

\let\setminus\smallsetminus
\let\grminus\setminus
\let\epsi\varepsilon
\let\emptyset\varnothing

\newcommand{\brac}[1]{{\left(#1\right)}}
\newcommand{\sbrac}[1]{{\left[#1\right]}}
\newcommand{\set}[1]{\left\{#1\right\}}
\newcommand{\norm}[1]{{\left|#1\right|}}
\newcommand{\floor}[1]{{\left\lfloor #1 \right\rfloor}}
\newcommand{\ceil}[1]{{\left\lceil #1 \right\rceil}}

\newcommand{\Oh}[1]{O\brac{#1}}
\newcommand{\oh}[1]{o\brac{#1}}
\newcommand{\Om}[1]{\Omega\brac{#1}}
\newcommand{\om}[1]{\omega\brac{#1}}
\newcommand{\Ot}[1]{\Theta\brac{#1}}

\newcommand{\Nat}{\mathbb{N}}

\newcommand{\skel}{\operatorname{skel}}
\newcommand{\ext}{\operatorname{ext}}
\newcommand{\inter}[1]{\operatorname{int}\sbrac{#1}}
\newcommand{\minn}{\operatorname{minn}}
\newcommand{\maxn}{\operatorname{maxn}}

\makeatletter
\newcommand\ie{i.e\@ifnextchar.{}{.\@}}
\newcommand\etc{etc\@ifnextchar.{}{.\@}}
\newcommand\etal{et~al\@ifnextchar.{}{.\@}}
\makeatother

\newcounter{Cases}
\newcounter{CasesSub}
\newcounter{CasesSubSub}
\newcounter{CasesAll}
\makeatletter
\newcommand\caselab[1]{\def\@currentlabel{#1}}
\makeatother
\newcommand{\caselabel}[1]{\refstepcounter{CasesAll}\caselab{#1}}
\newcommand{\Cases}{\setcounter{Cases}{0}\setcounter{CasesSub}{0}\setcounter{CasesSubSub}{0}}
\newcommand{\Case}[1]{\stepcounter{Cases}\setcounter{CasesSub}{0}\setcounter{CasesSubSub}{0}\medskip\noindent\textbf{Case~\arabic{Cases}.\,\,#1.\\\noindent}\caselabel{\arabic{Cases}}}
\newcommand{\Subcase}[1]{\stepcounter{CasesSub}\setcounter{CasesSubSub}{0}\smallskip\noindent\textbf{Case~\arabic{Cases}.\arabic{CasesSub}.\,\,#1.\\\noindent}\caselabel{\arabic{Cases}.\arabic{CasesSub}}}
\newcommand{\Subsubcase}[1]{\stepcounter{CasesSubSub}\noindent\textbf{Case~\arabic{Cases}.\arabic{CasesSub}.\arabic{CasesSubSub}.\,\,#1.\\\noindent}\caselabel{\arabic{Cases}.\arabic{CasesSub}.\arabic{CasesSubSub}}}

\newcommand{\slowcol}[1]{\mathring{{\rm s}}\brac{#1}}

\newcommand{\arxiv}[1]{\href{http://arxiv.org/abs/#1}{\texttt{arXiv:#1}}}
\def\C#1{\left|{#1}\right|}
\def\spo{\mathring{{\rm s}}}
\def\st{\colon\,}
\def\set#1{\{#1\}}
\def\esub{\subseteq}
\def\nul{\varnothing}
\def\FR{\frac}
\def\FL{\floor}
\def\CL{\ceil}

\def\NN{{\mathbb N}}

\def\CH{\binom}
\def\VEC#1#2#3{#1_{#2},\ldots,#1_{#3}}
\def\SE#1#2#3{\sum_{#1=#2}^{#3}}
\def\cG{{\mathcal G}}
\def\cH{{\mathcal H}}
\def\cP{{\mathcal P}}

\def\deg{d}
\def\pc{\gamma}
\def\tsty{\textstyle}

\begin{document}

%

\title{The Slow-coloring Game on Sparse Graphs: \\
$k$-Degenerate, Planar, and Outerplanar
}

\author{
Grzegorz Gutowski\footnotemark[1],\quad
Tomasz Krawczyk\footnotemark[1],\quad
Krzysztof Maziarz\footnotemark[1],\quad\\
Douglas B. West\footnotemark[2]\,\,\footnotemark[3],\quad
Micha{\l} Zaj\k{a}c\footnotemark[1],\quad
Xuding Zhu\footnotemark[2]
}

\maketitle

\vspace{-2pc}

\begin{abstract}
The \emph{slow-coloring game} is played by Lister and Painter on a graph $G$.
Initially, all vertices of $G$ are uncolored.  In each round, Lister marks a
nonempty set $M$ of uncolored vertices, and Painter colors a subset of $M$
that is independent in $G$.  The game ends when all vertices are colored.
The score of the game is the sum of the sizes of all sets marked by Lister.
The goal of Painter is to minimize the score, while Lister tries to maximize it.
We provide strategies for Painter on various classes of graphs whose vertices
can be partitioned into a bounded number of sets inducing forests, including
$k$-degenerate, acyclically $k$-colorable, planar, and outerplanar graphs.
For example, we show that on an $n$-vertex graph $G$, Painter can keep the
score to at most $\frac{3k+4}4n$ when $G$ is $k$-degenerate, $3.9857n$ when $G$
is acyclically $5$-colorable, $3n$ when $G$ is planar with a Hamiltonian dual,
$\frac{8n+3m}5$ when $G$ is $4$-colorable with $m$ edges (hence $3.4n$ when $G$
is planar), and $\frac73n$ when $G$ is outerplanar.
\end{abstract}

\renewcommand{\thefootnote}{\fnsymbol{footnote}}
\footnotetext[1]{
Computer Science Department, Faculty of Mathematics and Computer Science,
Jagiellonian University, Krak\'ow, Poland.  {\it Email:}
\{gutowski,krawczyk\}@tcs.uj.edu.pl, \{krzysztof.s.maziarz,emzajac\}@gmail.com.
Research of G. Gutowski was partially supported by National Science Center of
Poland, grant UMO-2016/21/B/ST6/02165.  Research of T. Krawczyk, K. Maziarz,
and M. Zaj\k{a}c was partially supported by National Science Center of Poland,
grant UMO-2015/17/B/ST6/01873.}
\footnotetext[2]{Department of Mathematics, Zhejiang Normal University,
Jinhua, China.  {\it Email: xdzhu@zjnu.edu.cn}.
Research of X. Zhu supported in part by CNSF 00571319.}
\footnotetext[3]{Department of Mathematics, University of Illinois, Urbana, IL,
USA.  {\it Email:} {dwest@math.uiuc.edu}.
Research of D.B. West supported by Recruitment Program of Foreign Experts, 1000
Talent Plan, State Administration of Foreign Experts Affairs, China.}
\renewcommand{\thefootnote}{\arabic{footnote}}

\renewcommand\thesubfigure{\roman{subfigure}}


\vspace{-1pc}

\section{Introduction}


The \emph{slow-coloring game}, introduced by Mahoney, Puleo, and
West~\cite{MPW}, models proper coloring of graphs in a scenario with
restrictions on the coloring process.  The game is played by \emph{Lister} and
\emph{Painter} on a graph $G$.  Initially, all vertices of $G$ are uncolored.
In each round, Lister marks a nonempty set $M$ of uncolored vertices of $G$
and scores $\norm{M}$.  Painter responds by selecting an independent set
$X \subseteq M$ to receive the next color.  The game ends when all vertices in
$G$ are colored.  The score is the sum of the sizes of the sets marked by
Lister.  Lister's goal is to maximize the score; Painter's goal is to minimize
it.  The result when both players play optimally (to ensure the best possible
score they can guarantee) is denoted by $\slowcol{G}$ and called the
\emph{sum-color cost} of $G$.

A \emph{proper coloring} of a graph $G$ assigns distinct colors to adjacent
vertices.  The \emph{chromatic number} $\chi(G)$ is the minimum number of
colors in a proper coloring ($G$ is \emph{$k$-colorable} if $\chi(G)\le k$).
The sets selected by Painter during the slow-coloring game together form a 
proper coloring of $G$, using many colors.  However, when $G$ is $k$-colorable,
every set $M\esub V(G)$ contains an independent set of size at least $|M|/k$,
which means that if Painter selects a largest independent subset then the score
on each round is at most $\chi(G)$ times the number of vertices colored on that
round.  Summing over all rounds yields $\spo(G)\le \chi(G)\C{V(G)}$.  Since in
the last round all remaining vertices are colored, the inequality is strict
unless $\chi(G)=1$.

Nevertheless, Wu~\cite{Wu} (presented in~\cite{MPW}) proved that the upper
bound $\spo(G)\le \chi(G)\C{V(G)}$ is asymptotically sharp; his general lower
bound for complete multipartite graphs shows that $\spo(G)\sim kn$ for the
complete $k$-partite $n$-vertex graph with part-sizes $\VEC r1k$ differing by
at most $1$.  Wu's general lower bound is $n+\sum_{i<j} u_{r_i}u_{r_j}$, where
$u_r=\max\{t\st\CH{t+1}2\le r\}\sim \sqrt{2r}$.  Our first general technique
yields an immediate upper bound for complete multipartite graphs that is very
close to this; see Corollary~\ref{kpart1} below.

Let $n=\C{V(G)}$.  Using the greedy Painter strategy suggested above, Mahoney,
Puleo, and West~\cite{MPW} improved the general upper bound of $\chi(G)n$
to $\max_{H\esub G}\FR{|V(H)|}{\alpha(H)}n$, where $\alpha(H)$ denotes the
maximum size of an independent set in $H$ (also
${\spo(G)}\ge \FR12(1+\FR{n}{\alpha(G)})n$).  They proved
$n+u_{n-1}\le\slowcol{G}\le \floor{\frac{3}{2}n}$ when $G$ is an $n$-vertex
tree.  The lower bound is exact for the star $K_{1,n-1}$, and the upper bound
is exact for the path $P_n$ (and other trees).  Puleo and West~\cite{PW}
extended this, characterizing the
trees achieving the upper and lower bounds and giving a linear-time algorithm
(and inductive formula) to compute $\spo(G)$ when $G$ is a tree.

Our general theme is to improve the bound $k\C{V(G)}$ for interesting special
classes of $k$-chromatic graphs (the coefficient $k$ cannot be improved for
general $k$-chromatic graphs).  Given a graph $G$ and $A\esub V(G)$, let $G[A]$
denote the subgraph of $G$ induced by $A$.  We present a general Painter
strategy to prove the following result.

\begin{theorem}\label{main}
Let $G$ be an $n$-vertex graph.  If $\spo(G[V_i])\le c_i|V_i|$ for
$1\le i\le t$, where $V(G)$ is the disjoint union of $\VEC V1t$, then
$\spo(G)\le \left(\sum_i\sqrt{c_i|V_i|}\right)^2$ and
$\spo(G)\le \left(\sum_i c_i\right)n$.
\end{theorem}

The strategy ignores edges joining $V_i$ and $V_j$.  They may all be present,
which explains why a very good upper bound for complete multipartite graphs is
an immediate corollary.

\begin{corollary}\label{kpart1}
If $G$ is an $n$-vertex complete $k$-partite graph with part-sizes
$\VEC r1k$, then $\spo(G)\le n+2\sum_{i<j}\sqrt{r_ir_j}$.
\end{corollary}

\noindent
Note that the bound simplifies to $kn$ when all part-sizes equal $r$.

Forests are $2$-colorable.  Because the trivial upper bound $\spo(G)\le 2n$ has
been improved to the optimal bound $\spo(G)\le\FL{\FR32 n}$ when $G$ is an
$n$-vertex forest, Theorem~\ref{main} provides useful bounds for graphs whose
vertices can be partitioned into $t$ sets inducing forests.  Such a graph
$G$ is $2t$-colorable, but Theorem~\ref{main} yields $\spo(G)\le\FR{3t}2n$.
We consider several classes of graphs admitting vertex partitions into a small
number of sets inducing forests.

A graph is \emph{$k$-degenerate} when every nonempty subgraph has a vertex of
degree at most $k$.  A graph is \emph{planar} if it can be drawn on the plane
so that edges intersect only at their endpoints, and it is \emph{outerplanar}
if it has a such a drawing with all vertices lying on the unbounded face.
A graph is \emph{acyclically $k$-colorable} if it has a proper $k$-coloring
with no $2$-colored cycle, meaning that the union of any two color classes
induces a forest.

Inductively, every $k$-degenerate graph is $(k+1)$-colorable.  Outerplanar
graphs are $2$-degenerate and hence $3$-colorable, and planar graphs are
$4$-colorable~\cite{AH,RSST}.  Acyclically $k$-colorable graphs by definition
are $k$-colorable.  In these classes, applying the general decomposition
results improves the trivial upper bound $\spo(G)\le\chi(G)|V(G)|$ as follows.

\begin{corollary}
Let $G$ be an $n$-vertex graph.

{\rm (a)} If $G$ is $k$-degenerate, then $\spo(G)\le \frac{3k+4}4n$ for even
$k$; $\spo(G)\le \frac{3k+3}{4}n$ for odd $k$.

{\rm (b)} If $G$ is acyclically $k$-colorable, then $\spo(G)\le \frac{3k}4 n$
for even $k$; $\spo(G)\le \frac{3k+1}4 n$ for odd $k$.

{\rm (c)} If $G$ is acyclically $k$-colorable and $k$ is odd, then 
$\spo(G)\le \frac1k[\sqrt{.75}(k-1)+1]^2 n$.  In particular, $\spo(G)\le3.9857$
for acyclically $5$-colorable graphs, which includes planar graphs.

{\rm (d)} If $G$ is a plane graph and the planar dual of $G$ has a spanning
cycle, then $\spo(G)\le 3n$.
\end{corollary}

Except for Corollary~\ref{kpart1}, the bounds are most likely not sharp, since
the arguments allow extra edges that in the special classes cannot all appear.
Our closest constructions arise from disjoint unions of complete graphs.  Note
that $\spo(H)\le\spo(G)$ when $H\esub G$ and that $\spo$ is additive under
disjoint union.  Also, the complete graph $K_n$ satisfies
$\spo(K_n)=\CH{n+1}2=\FR{n+1}2n$, and $K_n$ is $(n-1)$-degenerate.
Thus, among $k$-degenerate graphs, a graph $G$ consisting of disjoint copies of
$K_{k+1}$ satisfies $\spo(G)=(\FR12 k+1)n$.  Using $k=3$ for planar graphs
and $k=2$ for outerplanar graphs yields lower bounds of $\FR52n$ and $2n$,
respectively.

The results mentioned above appear in Section~\ref{sec:decomposition}.
In Section~\ref{sec:potentials} and beyond, we introduce a different technique
with a more complicated algorithm for Painter that yields better
upper bounds.  Using appropriate ``potential functions'' on graphs, it provides
the following bounds.

\begin{theorem}\label{thm:outerplanar}\label{thm:pla}
Let $G$ be an $n$-vertex graph with $m$ edges.

{\rm (a)} If $G$ is $4$-colorable, then $\spo(G)\le \frac{8n+3m}5$
(in particular, $\spo(G)\le 3.4n$ when $G$ is planar).

{\rm (b)} If $G$ is outerplanar, then $\spo(G)\le\FR73 n$.
\end{theorem}

\section{The Decomposition Method}\label{sec:decomposition}

We introduce a way to combine strategies for Painter on disjoint induced
subgraphs.

\begin{definition}\label{alpha-comb}
Let $\cP$ be a partition of the vertices of a graph $G$,
with $\cP=\{\VEC V1t\}$.  For $1\le i\le t$, let
$w_i=\sqrt{c_i|V_i|}/\sum_j\sqrt{c_j|V_j|}$, where $\spo(G[V_i])\le c_i|V_i|$.
The \emph{$\cP$-composite} strategy for Painter on $G$ is as follows.  When
Lister marks a set $M$, Painter chooses any index $i$ such that
$\C{M\cap V_i}/|M|\ge w_i$ and responds to the move $M\cap V_i$ according to an
optimal strategy on $G[V_i]$, ignoring the rest of $M$.  Such an index exists
because $\sum_i w_i=1$.  To be deterministic, Painter may choose the least such
index.
\end{definition}

\begin{theorem}\label{decomp}\label{lem:decomposition}
Let $\cP$ be a partition of the vertex set of a graph $G$ into sets
$\VEC V1t$.  If $\spo(G[V_i])\le c_i|V_i|$, then
\begin{equation}\label{eq1}
\spo(G)\le \Bigl(\sum_i\sqrt{c_i|V_i|}\Bigr)^2,
\end{equation}
and
\begin{equation}\label{eq2}
\spo(G)\le \Bigl(\sum_i c_i\Bigr)|V(G)|.
\end{equation}
\end{theorem}

\begin{proof}
Our main task is to prove~\eqref{eq1}, from which~\eqref{eq2} follows by
numerical arguments.

We use the $\cP$-composite strategy for Painter, with weights $\VEC w1t$ as in
Definition~\ref{alpha-comb}.  When Lister marks $M$ and Painter choose the
index $i$ to play on $G[V_i]$, we have $\C{M\cap V_i}\ge w_i|M|$, and hence
$|M|\le \C{M\cap V_i}/w_i$.

Rounds played in $G[V_i]$ form a game on $G[V_i]$ played optimally by Painter.
Therefore, over those rounds $\C{M\cap V_i}$ sums to at most $c_i|V_i|$ and
$|M|$ sums to at most $c_i|V_i|/w_i$.  Over all rounds, 
$\spo(G)\le \sum_i (c_i|V_i|/w_i) =\bigl(\sum_i\sqrt{c_i|V_i|}\bigr)^2$,
completing the proof of~\eqref{eq1}

%

Next, the Arithmetic-Geometric Mean Inequality yields
$2\sqrt{c_i\C{V_j}c_j\C{V_i}}\le c_i\C{V_j}+c_j\C{V_i}$.  Using this in
the expansion of~\eqref{eq1} yields
$$
\begin{array}{rcl}
\slowcol{G} &\le & \bigl(\sum_i \sqrt{c_i\C{V_i}}\bigr)^2
~=~ \sum_i c_i\C{V_i} + \sum_{i<j}2\sqrt{c_i\C{V_i}c_j\C{V_j}} \\
&=&\vbox to17pt{} \sum_i c_i\C{V_i}+\sum_{i<j}2\sqrt{c_i\C{V_j}c_j\C{V_i}}
~\le~ \sum_i c_i\C{V_i} + \sum_{i<j}(c_i\C{V_j}+c_j\C{V_i})\\
&=& \vbox to17pt{}\sum_i c_i\sum_j\C{V_j} ~=~ \left(\sum_i c_i\right) \C{V(G)}.
\end{array}
$$

\vspace{-2pc}
\end{proof}

\bigskip
As an immediate corollary, consider the complete $k$-partite graph with
part-sizes $\VEC r1k$.  The lower bound by Wu~\cite{Wu} is
$\sum_i r_i+\sum_{i<j} u_{r_i}u_{r_j}$, where $u_r=\max\{t\st\CH{t+1}2\le r\}$.
Note that
$u_r=\FL{(-1+\sqrt{1+8r})/2}\approx\sqrt{2r}$.  Actually, $u_r$ is a bit
smaller than $\sqrt{2r}$, but within a small constant.  Hence our corollary is
very close to the lower bound.  Note that~\eqref{eq2} gives only the trivial
bound $k|V(G)|$.

\begin{corollary}\label{kpart}
If $G$ is an $n$-vertex complete $k$-partite graph with part-sizes
$\VEC r1k$, then $\spo(G)\le n+2\sum_{i<j}\sqrt{r_ir_j}$.
\end{corollary}
\begin{proof}
Use the maximal independent sets as the parts $\VEC V1k$ in a partition of
$V(G)$.  Since $\spo(G[V_i])=\C{V_i}$, we have $c_i=1$ for all $i$.
Hence expanding~\eqref{eq1} yields the claim.
\end{proof}

\bigskip
The case $k=2$ of this upper bound was proved in~\cite{MPW}.  Wu~\cite{Wu}
actually proved the better upper bound
$n+\sum_{i<j}\sqrt{2r_i-1}\sqrt{2r_j-1}$ by a more difficult argument.
That bound has the virtue of being exact for some stars.

Theorem~\ref{decomp} and the bound of~\cite{MPW} for forests can be combined
to obtain upper bounds when the vertices of a graph can be partitioned into a
small number of sets inducing forests.

\begin{theorem}\label{arb}
If the vertex set of a graph $G$ can be partitioned into $t$ sets inducing 
forests, then $\spo(G)\le \FR{3t}2\C{V(G)}$.  If it can be partitioned into
one independent set and $t-1$ sets inducing forests, then 
$\spo(G)\le \FR{3t-1}2\C{V(G)}$.
\end{theorem}
\begin{proof}
Let $\VEC V1t$ be the given vertex partition.  To apply Theorem~\ref{decomp},
we set $c_i=3/2$ for each $i$, except $c_t=1$ in the second case.  The claim
then follows immediately from~\eqref{eq2}.
\end{proof}

By induction on the number of vertices, every $k$-degenerate graph has a
vertex ordering in which every vertex has at most $k$ earlier neighbors; call
this a \emph{$k$-ordering}.  The following lemma is well-known.

\begin{lemma}\label{degenfor}
Let $G$ be a $k$-degenerate graph.  For $\VEC k1t$ such that
$\SE i1t (k_i+1)\ge k+1$, there is a partition of $V(G)$ into sets $\VEC V1t$
such that $G[V_i]$ is $k_i$-degenerate, for each $i$.
\end{lemma}
\begin{proof}
Such a partition is produced iteratively by considering the vertices in the
order of a $k$-ordering.  When we reach the $j$th vertex, it can be placed
safely in one of the sets, because having more than $k_i$ earlier neighbors in
each evolving set $V_i$ requires the vertex to have more than $k$ neighbors in
$G$ that are earlier in the $k$-ordering.
\end{proof}

This idea of Lemma~\ref{degenfor} was used by Chartrand and Kronk~\cite{CK} to
show that a $k$-degenerate graph decomposes into $\CL{(k+1)/2}$ forests
(forests are the $1$-degenerate graphs).  We use $(k+1)/2$ forests when $k$ is
odd, $k/2$ forests plus one independent set when $k$ is even.  The corollary
then follows from Theorem~\ref{arb}.  Similarly, when a graph is acyclically
$k$-colorable, again the conditions of Theorem~\ref{arb} apply, since
grouping color classes in pairs provides a decomposition of the vertex set
into sets inducing forests.

\begin{corollary}\label{degen}
Let $G$ be an $n$-vertex graph.  If $G$ is $k$-degenerate, 
then $\spo(G)\le \frac{3k+4}4n$ for even $k$ and 
$\spo(G)\le \frac{3k+3}{4}n$ for odd $k$.
If $G$ is acyclically $k$-colorable, then $\spo(G)\le \frac{3k}4 n$ for
even $k$ and $\spo(G)\le \frac{3k+1}4 n$ for odd $k$.
\end{corollary}

Although this upper bound for $k$-degenerate graphs is halfway between the
trivial upper bound and the trivial lower bound, it does not seem strong,
because the argument allows all edges joining vertices in distinct forests in
the partition, but having all such edges would contradict the degree
requirements in the full $k$-degenerate graph.

In general, when the coefficients $\VEC c1t$ are the same and we do not know
the sizes of $\VEC V1t$, the bound from~\eqref{eq1} does not improve
on~\eqref{eq2}.  The reason is that the square-root function is concave, and
hence the bound in~\eqref{eq1} is largest when the parts have equal size.  In
that case the bound $(\sum_i \sqrt{c_i|V_i|})^2$ becomes $(t\sqrt{cn/t})^2$,
which equals $ctn$.

When $V(G)$ is partitioned into one independent set and $t-1$ sets inducing
forests (such as when $G$ is $(2t-2)$-degenerate or acyclically
$(2t-1)$-colorable), we can improve on the bound in Theorem~\ref{arb} by
using a result intermediate between~\eqref{eq1} and~\eqref{eq2} that takes
advantage of the difference between the coefficients $3/2$ for the forests
and $1$ for the independent set.

\begin{theorem}\label{gamma}
Let $A$ and $B$ partition the vertex set of a graph $G$.  If
$\spo(G\sbrac{A})\le c_A\norm{A}$ and $\spo(G\sbrac{B})\le c_B\norm{B}$, and
$\gamma$ is a constant satisfying
$\frac{c_A}{c_A + c_B}\le\gamma\le \frac{\norm{A}}{\norm{A}+\norm{B}}$,
then
$$\slowcol{G} \le \brac{\sqrt{c_A \gamma} + \sqrt{c_B \brac{1 - \gamma}}}^2\norm{V(G)}\text{.}$$
\end{theorem}


\begin{proof}
The condition for the existence of such $\gamma$ is equivalent to
$c_A\C B\le c_B\C A$; we may label $A$ and $B$ so that this holds and there is
a claim to prove.  Let $\beta= \FR{\C A}{\C{V(G)}}$, and let
$g(x)=\sqrt{c_Ax}+\sqrt{c_B(1-x)}$ for $0\le x\le 1$.  By~\eqref{eq1}, we have
$$
\spo(G)\le \brac{\sqrt{c_A\norm{A}}+\sqrt{c_B\norm{B}}}^2 = g(\beta)^2\C{V(G)}.
$$
Since $g'(x)=\FR12\brac{\sqrt{{c_A}/{x}} - \sqrt{{c_B}/{(1-x)}}},$
we have $g'(x) \le 0$ if and only if $x \ge \frac{c_A}{c_A + c_B}$.
Also $g(x)\ge0$ on $[0,1]$, so we conclude that $g(x)$ and $(g(x))^2$
are nonincreasing on the interval $[{\frac{c_A}{c_A + c_B}, 1}]$.
Since this interval contains the interval $[{\frac{c_A}{c_A+c_B},\beta}]$,
which contains $\gamma$, we have
$$\slowcol{G} \le g(\beta)^2\norm{V(G)} \le g(\gamma)^2 \norm{V(G)} =
\brac{\sqrt{c_A \gamma} + \sqrt{c_B \brac{1 - \gamma}}}^2\norm{V(G)}.$$

\vspace{-2.2pc}
\end{proof}

\bigskip
Since $g(x)$ is nonincreasing, the bound from Theorem~\ref{gamma} is
strongest when $\gamma=|A|/|V(G)|$, where it is the same as~\eqref{eq1},
and weakest when $\gamma=c_A/(c_A+c_B)$, where it is the same as~\eqref{eq2}.
When we know the coefficients $c_A$ and $c_B$ but do not know $|A|$ and $|B|$,
we may still be able to improve on~\eqref{eq2} if we can bound $|A|/|V(G)|$
from below.

For a $k$-degenerate graph with $k$ even, we can partition the vertices into
$k/2$ sets inducing forests and one independent set, but in doing this we
cannot control the size of the independent set.  For an acyclically
$k$-colorable graph with $k$ odd, we obtain the coloring first and combine the
classes arbitrarily in pairs to form sets inducing forests, so in this case we
can require the independent set to be the smallest of the classes and play the
role of $B$.  Since our previous bound is $2.5n$ for acyclically $3$-colorable
graphs, $4n$ for acyclically $5$-colorable graphs, and $\FR14(3k+1)n$ in
general, we obtain an improvement.

\begin{corollary}\label{acyc}
If $G$ is an acyclically $k$-colorable graph with $n$ vertices, where $k$ is
odd, then $\spo(G)\le \frac1k[\sqrt{.75}(k-1)+1]^2 n$.  In particular, the
coefficient on $n$ is less than $2.4881$ when $k=3$ and less than $3.9857$ when
$k=5$.
\end{corollary}
\begin{proof}
Let $B$ be the smallest color class in an acyclic $k$-coloring, and let $A$ be
the union of the $k-1$ largest color classes.  Since $A$ can be partitioned into
$(k-1)/2$ sets inducing forests in $G$, we have $\spo(G[A])\le \FR34(k-1)\C{A}$.
Since $\spo(G[B])=\C{B}$, we have $c_B=1$ and $c_A=\FR34(k-1)$ in the notation
of Theorem~\ref{gamma}.

The choice of $B$ guarantees $\C B\le n/k$ and $\C A\ge (k-1)n/k$.  Hence
$c_A|B|\le \FR34(k-1)n/k$ and $c_B|A|\ge (k-1)n/k$, so Theorem~\ref{gamma}
applies.  Since $\frac{c_A}{c_A+c_B}\le \FR{k-1}k\le \frac{|A|}{|A|+|B|}$,
we can set $\gamma=(k-1)/k$.  By Theorem~\ref{gamma},
$$
\spo(G)\le\left(\sqrt{\FR34(k-1)\FR{k-1}k}+\sqrt{1\cdot\FR1k}\right)^2n
=\frac1k[\sqrt{.75}(k-1)+1]^2 n.
$$

\vspace{-2pc}
\end{proof}

\bigskip
Grouping of color classes can be applied more generally, although doing so
does not yet improve on Corollary~\ref{acyc}.  Let $\cH$ be a hereditary family
of graphs.  A graph is $\cH$ $r$-colorable if its vertices can be partitioned
into $r$ sets inducing subgraphs in $\cH$.

\begin{proposition}\label{lem:decomposition_acyclic}
For $r\in\NN$, let $c_r$ be a constant such that $\spo(G) \le c_r|V(G)|$
whenever $G$ is $\cH$ $r$-colorable.  If $G$ is $\cH$ $(p+q)$-colorable, then
$$\spo(G) \le \frac{(\sqrt{pc_p}+\sqrt{qc_q})^2}{p+q}|V(G)|.$$
\end{proposition}

\begin{proof}
Consider an $\cH$ $(p+q)$-coloring of $G$ with colors $1,\ldots,p+q$, indexed
in nonincreasing order of the sizes of the color classes.  Note that
$\FR{c_p}{c_p+c_q}+\FR{c_q}{c_q+c_p}=1=\FR{p}{p+q}+\FR{q}{q+p}$.
Therefore, exactly one of $\FR{c_p}{c_p+c_q}\le \FR{p}{p+q}$ and 
$\FR{c_q}{c_q+c_p}<\FR{q}{q+p}$ holds.  By symmetry, we may assume
$\FR{c_p}{c_p+c_q}\le \FR{p}{p+q}$.

Let $P$ denote the set of vertices having colors $1,\ldots,p$, and let 
$Q=V(G)-P$.  Since $\chi_a(G\sbrac{P}) \le p$ and $\chi_a(G\sbrac{Q}) \le q$,
we have $\spo(G[P])\le c_p|P|$ and $\spo(G[Q])\le c_q|Q|$.  By the indexing
of the color classes, $\FR{|P|}{|Q|}\ge\FR pq$, and hence
$\frac{\norm{P}}{\norm{P}+\norm{Q}} \ge \frac{p}{p+q}$.

With $\gamma = \frac{p}{p + q}$, we thus have
$\frac{c_p}{c_p+c_q} \le \gamma \le \frac{\norm{P}}{\norm{P}+\norm{Q}}$.
Now we apply Theorem~\ref{gamma} using sets $P$ and $Q$ and parameter $\gamma$
to obtain
$$
\slowcol{G}\le
\left(\sqrt{c_p \frac{p}{p+q}}+\sqrt{c_q\frac{q}{p+q}}\right)^2\norm{V(G)} 
=\frac{\brac{\sqrt{pc_p} + \sqrt{qc_q}}^2}{p + q}\norm{V(G)}\text{.}
$$

\vspace{-2.7pc}
\end{proof}

\bigskip
\bigskip

In fact, Proposition~\ref{lem:decomposition_acyclic} provides a recursive
upper bound for the sequence $\{c_r\}_{r\ge1}$, by minimizing that bound
over $\{p,q\}$ such that $p+q=r$.

Borodin~\cite{B79} proved that planar graphs are acyclically $5$-colorable, so
the bound we have given for acyclically $5$-colorable graphs holds also for all
planar graphs.  It is slightly better than the trivial lower bound of
$4|V(G)|$ implied by the Four Color Theorem~\cite{AH}.  In Section~\ref{sec:pla}
we will improve the general upper bound for planar graphs to $3.4|V(G)|$.  Our
results above allow us to improve the bound further for a special class of
planar graphs.  A graph is {\it Hamiltonian} if it has a spanning cycle.

\bigskip
\begin{proposition}\label{doubletree}
If $G$ is a plane graph whose vertices can be partitioned into sets $A$ and $B$
such that $\spo(G)\le c|A|$ and $\spo(G[B])\le c|B|$, then $\spo(G)\le2c|V(G)|$.
In particular, $\spo(G)\le 3|V(G)|$ when the dual graph of $G$ is Hamiltonian.
\end{proposition}
\begin{proof}
The first claim is a special case of~\eqref{eq2}.  For the second, it is
well known that the vertices of a plane graph $G$ can be partitioned into sets
$A$ and $B$ inducing forests if and only if the dual graph of $G$ is Hamiltonian
(Stein~\cite{St} proved this for plane graphs whose faces are all triangles).
Hence in this case the hypothesis holds with $c=\FR32$.  
\end{proof}

\bigskip
We know of no $n$-vertex planar graph $G$ with $\spo(G)>5n/2$; equality holds
for graphs whose components are copies of $K_4$.  Note also that since planar
graphs are $4$-colorable, the vertex set of any planar graph
can be partitioned into two sets inducing bipartite graphs.  Hence bounding
$\spo(G)$ for planar bipartite graphs may also be of interest.  Here there is
a nontrivial construction.  Puleo (private communication) showed that
$\spo(G)=1.75n-1$ when $G$ is the cartesian product of a $4$-cycle and a path;
this graph $G$ is planar and bipartite.


\section{The Potential Method}\label{sec:potentials} 

We introduce another technique to prove upper bounds for planar and outerplanar
graphs.  Painter uses a ``potential function'' $\Phi$ on the vertices and edges
of a graph that can be thought of as summing ``potential'' contributions to the
score in the remaining game.  The total potential is the sum of these
contributions.  The contributions are different in our two applications, but we
explain the technique first as a common generalization that can be applied to
other families of sparse graphs.

\begin{definition}\label{pot}
Let $\phi_G$ be a function assigning a positive real number (called 
``potential'') to each vertex and edge of a graph $G$.  For $G$ in a given
class $\cG$ of graphs, each edge will have potential $\pc$, and the vertex
potentials will be defined later in such a way that $\phi_H(x)\le \phi_G(x)$
when $x\in V(H)$ and $H\esub G$ (that is, the potential function is {\it
monotone}).  For a graph $G$, define the total potential $\Phi(G)$ by
${\Phi(G) = \sum_{x \in V(G) \cup E(G)} \phi_G(x)}$.
\end{definition}

The goal of this method is to prove $\spo(G)\le\Phi(G)$ for graphs $G$ in a
given hereditary class $\cG$ (closed under taking induced subgraphs).  Since
the potential has a contribution that is linear in the number of edges,
applying this to a family of $k$-colorable graphs can give an improvement
over the trivial bound only when the number of edges is at most linear in the
number of vertices, which holds for planar and outerplanar graphs.

When Lister marks a set $M$ in a graph $G$ in $\cG$, Painter will seek an
independent set $X\esub M$ such that $|M|\le\Phi(G)-\Phi(G-X)$.  That is, the
total score in the current round should be at most the loss in potential by
coloring $X$.  Since the potential is reduced to $0$ when the game is over,
always being able to find such a set $X$ yields $\spo(G)\le\Phi(G)$.  To
consider $X$, we define the ``utility'' of $X$ relative to $M$, which we also
split into contributions from the various vertices of $G$.  Let $\deg_G(v)$
denote the degree of vertex $v$ in a graph $G$.  

\begin{definition}\label{util}
Let $\cG$ be a hereditary class $\cG$ on which a monotone potential function is
defined, with $\pc$ being the potential of each edge.  For each $M\esub V(G)$
and each independent set $X$ contained in $M$, define the {\it utility} of $X$
by $u(X)=\Phi(G)-\Phi(G-X)-|M|$.  (We seek $X$ such that $u(X)\ge0$.)
Apportion $u(X)$ among the vertices of $G$ by letting
$$u_X(v) = \begin{cases}
           \phi_G(v) + \pc \deg_G(v) -1 & \text{if $v \in X$,} \\
           \phi_G(v) - \phi_{G-X}(v) - 1& \text{if $v \in M - X$,} \\
           \phi_G(v) - \phi_{G-X}(v) & \text{if $v \in V(G) - M$.}
          \end{cases}
$$
\end{definition}

\medskip
\begin{lemma}\label{prop:util}
For $X\esub M\esub V(G)$ with $X$ independent, $\sum_{v\in V(G)} u_X(v)=u(X)$.
Also, $u_X(v)\ge -1$ when the potential function is monotone.
\end{lemma}
\begin{proof}
The terms $-1$ for $v\in X$ and $v\in M-X$ count $|M|$ negatively.  The
contribution of edges to $\Phi(G)-\Phi(G-X)$ is $\pc$ for every edge incident
to $X$; since $X$ is independent, this equals $\sum_{v\in X}\pc d_G(v)$.
When $v\in X$, vertex $v$ contributes $\phi_G(v)$ to $\Phi(G)$ and nothing
to $\Phi(G-X)$; otherwise, $v$ contributes $\phi_G(v)-\phi_{G-X}(v)$ to
$\Phi(G)-\Phi(G-X)$.

Since always $\phi_G(v)\ge 0$ and $\phi_G(v)-\phi_{G-X}(v)\ge 0$,
always $u_X(v)\ge -1$.
\end{proof}

When $G$ belongs to a family $\cG$ of $k$-colorable graphs and $M$ is a marked
set in $G$, we will cover $M$ using $k$ independent sets $\VEC {V'}1k$ such
that $\sum_i u(V'_i)\ge 0$ (some vertices may appear in more than one set).
Hence at least one set has nonnegative utility and can be chosen as the desired
play for Painter.  We may assume that $G[M]$ is connected, that is, $M$ is
a {\it connected set}.
In order to produce $\VEC {V'}1k$, we begin with a special $k$-coloring of
$G[M]$. 

\begin{definition}\label{components}
Given a connected set $M$ in a $k$-colorable graph $G$, let $T$ be the set of
vertices in $M$ that lie in no cycle in $G[M]$, and let $S=M-T$.  Call the
components of $G[T]$ {\it tree-components} and the components of $G[S]$
{\it cycle-components}.  A {\it good $k$-coloring} of $G[M]$ is a proper
$k$-coloring such that every tree component is $2$-colored and the color on
every vertex having a neighbor in a tree component is one of the two colors
assigned to that component.
\end{definition}

In Figure~\ref{fig:4col}, $T_1$ and $T_2$ are tree-components, $C_1$ and $C_2$
are cycle-components, and the marked vertices have degree at most $2$ in
$G[M]$.  We will focus on these later.

\begin{figure}[h!]
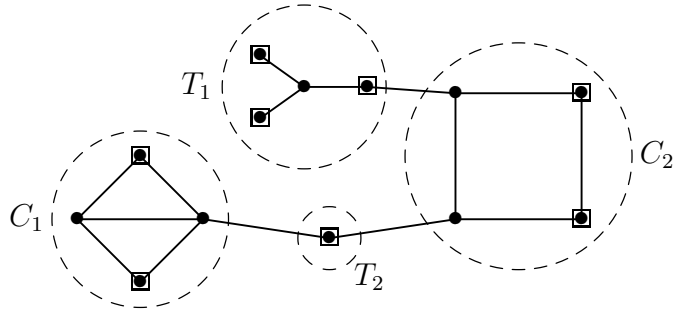

\gpic{
\expandafter\ifx\csname graph\endcsname\relax \csname newbox\endcsname\graph\fi
\expandafter\ifx\csname graphtemp\endcsname\relax \csname newdimen\endcsname\graphtemp\fi
\setbox\graph=\vtop{\vskip 0pt\hbox{%
    \graphtemp=.5ex\advance\graphtemp by 1.122in
    \rlap{\kern 0.264in\lower\graphtemp\hbox to 0pt{\hss $\bu$\hss}}%
    \graphtemp=.5ex\advance\graphtemp by 1.452in
    \rlap{\kern 0.594in\lower\graphtemp\hbox to 0pt{\hss $\bu$\hss}}%
    \graphtemp=.5ex\advance\graphtemp by 0.792in
    \rlap{\kern 0.594in\lower\graphtemp\hbox to 0pt{\hss $\bu$\hss}}%
    \graphtemp=.5ex\advance\graphtemp by 1.122in
    \rlap{\kern 0.924in\lower\graphtemp\hbox to 0pt{\hss $\bu$\hss}}%
    \graphtemp=.5ex\advance\graphtemp by 1.221in
    \rlap{\kern 1.584in\lower\graphtemp\hbox to 0pt{\hss $\bu$\hss}}%
    \graphtemp=.5ex\advance\graphtemp by 1.122in
    \rlap{\kern 2.244in\lower\graphtemp\hbox to 0pt{\hss $\bu$\hss}}%
    \graphtemp=.5ex\advance\graphtemp by 1.122in
    \rlap{\kern 2.904in\lower\graphtemp\hbox to 0pt{\hss $\bu$\hss}}%
    \graphtemp=.5ex\advance\graphtemp by 0.462in
    \rlap{\kern 2.904in\lower\graphtemp\hbox to 0pt{\hss $\bu$\hss}}%
    \graphtemp=.5ex\advance\graphtemp by 0.462in
    \rlap{\kern 2.244in\lower\graphtemp\hbox to 0pt{\hss $\bu$\hss}}%
    \special{pn 11}%
    \special{pa 264 1122}%
    \special{pa 594 1452}%
    \special{fp}%
    \special{pa 594 1452}%
    \special{pa 924 1122}%
    \special{fp}%
    \special{pa 924 1122}%
    \special{pa 594 792}%
    \special{fp}%
    \special{pa 594 792}%
    \special{pa 264 1122}%
    \special{fp}%
    \special{pa 264 1122}%
    \special{pa 924 1122}%
    \special{fp}%
    \special{pa 924 1122}%
    \special{pa 1584 1221}%
    \special{fp}%
    \special{pa 1584 1221}%
    \special{pa 2244 1122}%
    \special{fp}%
    \special{pa 2244 1122}%
    \special{pa 2904 1122}%
    \special{fp}%
    \special{pa 2904 1122}%
    \special{pa 2904 462}%
    \special{fp}%
    \special{pa 2904 462}%
    \special{pa 2244 462}%
    \special{fp}%
    \special{pa 2244 462}%
    \special{pa 2244 1122}%
    \special{fp}%
    \graphtemp=.5ex\advance\graphtemp by 1.452in
    \rlap{\kern 0.594in\lower\graphtemp\hbox to 0pt{\hss $\marker$\hss}}%
    \graphtemp=.5ex\advance\graphtemp by 0.792in
    \rlap{\kern 0.594in\lower\graphtemp\hbox to 0pt{\hss $\marker$\hss}}%
    \graphtemp=.5ex\advance\graphtemp by 1.221in
    \rlap{\kern 1.584in\lower\graphtemp\hbox to 0pt{\hss $\marker$\hss}}%
    \graphtemp=.5ex\advance\graphtemp by 1.122in
    \rlap{\kern 2.904in\lower\graphtemp\hbox to 0pt{\hss $\marker$\hss}}%
    \graphtemp=.5ex\advance\graphtemp by 0.462in
    \rlap{\kern 2.904in\lower\graphtemp\hbox to 0pt{\hss $\marker$\hss}}%
    \special{pn 8}%
    \special{ar 594 1122 462 462 -0.071429 0.071429}%
    \special{ar 594 1122 462 462 -0.333228 -0.190371}%
    \special{ar 594 1122 462 462 -0.595027 -0.452170}%
    \special{ar 594 1122 462 462 -0.856827 -0.713970}%
    \special{ar 594 1122 462 462 -1.118626 -0.975769}%
    \special{ar 594 1122 462 462 -1.380426 -1.237568}%
    \special{ar 594 1122 462 462 -1.642225 -1.499368}%
    \special{ar 594 1122 462 462 -1.904024 -1.761167}%
    \special{ar 594 1122 462 462 -2.165824 -2.022967}%
    \special{ar 594 1122 462 462 -2.427623 -2.284766}%
    \special{ar 594 1122 462 462 -2.689422 -2.546565}%
    \special{ar 594 1122 462 462 -2.951222 -2.808365}%
    \special{ar 594 1122 462 462 -3.213021 -3.070164}%
    \special{ar 594 1122 462 462 -3.474821 -3.331963}%
    \special{ar 594 1122 462 462 -3.736620 -3.593763}%
    \special{ar 594 1122 462 462 -3.998419 -3.855562}%
    \special{ar 594 1122 462 462 -4.260219 -4.117362}%
    \special{ar 594 1122 462 462 -4.522018 -4.379161}%
    \special{ar 594 1122 462 462 -4.783818 -4.640960}%
    \special{ar 594 1122 462 462 -5.045617 -4.902760}%
    \special{ar 594 1122 462 462 -5.307416 -5.164559}%
    \special{ar 594 1122 462 462 -5.569216 -5.426359}%
    \special{ar 594 1122 462 462 -5.831015 -5.688158}%
    \special{ar 594 1122 462 462 -6.092814 -5.949957}%
    \special{ar 1584 1221 165 165 -0.200000 0.200000}%
    \special{ar 1584 1221 165 165 -0.985398 -0.585398}%
    \special{ar 1584 1221 165 165 -1.770796 -1.370796}%
    \special{ar 1584 1221 165 165 -2.556194 -2.156194}%
    \special{ar 1584 1221 165 165 -3.341593 -2.941593}%
    \special{ar 1584 1221 165 165 -4.126991 -3.726991}%
    \special{ar 1584 1221 165 165 -4.912389 -4.512389}%
    \special{ar 1584 1221 165 165 -5.697787 -5.297787}%
    \special{ar 2574 792 594 594 -0.055556 0.055556}%
    \special{ar 2574 792 594 594 -0.251905 -0.140794}%
    \special{ar 2574 792 594 594 -0.448255 -0.337144}%
    \special{ar 2574 792 594 594 -0.644604 -0.533493}%
    \special{ar 2574 792 594 594 -0.840954 -0.729843}%
    \special{ar 2574 792 594 594 -1.037303 -0.926192}%
    \special{ar 2574 792 594 594 -1.233653 -1.122542}%
    \special{ar 2574 792 594 594 -1.430002 -1.318891}%
    \special{ar 2574 792 594 594 -1.626352 -1.515241}%
    \special{ar 2574 792 594 594 -1.822701 -1.711590}%
    \special{ar 2574 792 594 594 -2.019051 -1.907940}%
    \special{ar 2574 792 594 594 -2.215401 -2.104289}%
    \special{ar 2574 792 594 594 -2.411750 -2.300639}%
    \special{ar 2574 792 594 594 -2.608100 -2.496988}%
    \special{ar 2574 792 594 594 -2.804449 -2.693338}%
    \special{ar 2574 792 594 594 -3.000799 -2.889688}%
    \special{ar 2574 792 594 594 -3.197148 -3.086037}%
    \special{ar 2574 792 594 594 -3.393498 -3.282387}%
    \special{ar 2574 792 594 594 -3.589847 -3.478736}%
    \special{ar 2574 792 594 594 -3.786197 -3.675086}%
    \special{ar 2574 792 594 594 -3.982546 -3.871435}%
    \special{ar 2574 792 594 594 -4.178896 -4.067785}%
    \special{ar 2574 792 594 594 -4.375245 -4.264134}%
    \special{ar 2574 792 594 594 -4.571595 -4.460484}%
    \special{ar 2574 792 594 594 -4.767945 -4.656833}%
    \special{ar 2574 792 594 594 -4.964294 -4.853183}%
    \special{ar 2574 792 594 594 -5.160644 -5.049533}%
    \special{ar 2574 792 594 594 -5.356993 -5.245882}%
    \special{ar 2574 792 594 594 -5.553343 -5.442232}%
    \special{ar 2574 792 594 594 -5.749692 -5.638581}%
    \special{ar 2574 792 594 594 -5.946042 -5.834931}%
    \special{ar 2574 792 594 594 -6.142391 -6.031280}%
    \graphtemp=.5ex\advance\graphtemp by 0.429in
    \rlap{\kern 1.782in\lower\graphtemp\hbox to 0pt{\hss $\bu$\hss}}%
    \graphtemp=.5ex\advance\graphtemp by 0.429in
    \rlap{\kern 1.452in\lower\graphtemp\hbox to 0pt{\hss $\bu$\hss}}%
    \graphtemp=.5ex\advance\graphtemp by 0.264in
    \rlap{\kern 1.221in\lower\graphtemp\hbox to 0pt{\hss $\bu$\hss}}%
    \graphtemp=.5ex\advance\graphtemp by 0.594in
    \rlap{\kern 1.221in\lower\graphtemp\hbox to 0pt{\hss $\bu$\hss}}%
    \special{pn 11}%
    \special{pa 1452 429}%
    \special{pa 1782 429}%
    \special{fp}%
    \special{pa 1452 429}%
    \special{pa 1221 264}%
    \special{fp}%
    \special{pa 1452 429}%
    \special{pa 1221 594}%
    \special{fp}%
    \special{pa 1782 429}%
    \special{pa 2244 462}%
    \special{fp}%
    \special{pn 8}%
    \special{ar 1452 429 429 429 -0.076923 0.076923}%
    \special{ar 1452 429 429 429 -0.338722 -0.184876}%
    \special{ar 1452 429 429 429 -0.600522 -0.446676}%
    \special{ar 1452 429 429 429 -0.862321 -0.708475}%
    \special{ar 1452 429 429 429 -1.124121 -0.970274}%
    \special{ar 1452 429 429 429 -1.385920 -1.232074}%
    \special{ar 1452 429 429 429 -1.647719 -1.493873}%
    \special{ar 1452 429 429 429 -1.909519 -1.755673}%
    \special{ar 1452 429 429 429 -2.171318 -2.017472}%
    \special{ar 1452 429 429 429 -2.433118 -2.279271}%
    \special{ar 1452 429 429 429 -2.694917 -2.541071}%
    \special{ar 1452 429 429 429 -2.956716 -2.802870}%
    \special{ar 1452 429 429 429 -3.218516 -3.064670}%
    \special{ar 1452 429 429 429 -3.480315 -3.326469}%
    \special{ar 1452 429 429 429 -3.742115 -3.588268}%
    \special{ar 1452 429 429 429 -4.003914 -3.850068}%
    \special{ar 1452 429 429 429 -4.265713 -4.111867}%
    \special{ar 1452 429 429 429 -4.527513 -4.373667}%
    \special{ar 1452 429 429 429 -4.789312 -4.635466}%
    \special{ar 1452 429 429 429 -5.051111 -4.897265}%
    \special{ar 1452 429 429 429 -5.312911 -5.159065}%
    \special{ar 1452 429 429 429 -5.574710 -5.420864}%
    \special{ar 1452 429 429 429 -5.836510 -5.682663}%
    \special{ar 1452 429 429 429 -6.098309 -5.944463}%
    \graphtemp=.5ex\advance\graphtemp by 0.429in
    \rlap{\kern 1.782in\lower\graphtemp\hbox to 0pt{\hss $\marker$\hss}}%
    \graphtemp=.5ex\advance\graphtemp by 0.264in
    \rlap{\kern 1.221in\lower\graphtemp\hbox to 0pt{\hss $\marker$\hss}}%
    \graphtemp=.5ex\advance\graphtemp by 0.594in
    \rlap{\kern 1.221in\lower\graphtemp\hbox to 0pt{\hss $\marker$\hss}}%
    \graphtemp=.5ex\advance\graphtemp by 1.122in
    \rlap{\kern 0.000in\lower\graphtemp\hbox to 0pt{\hss $C_1$\hss}}%
    \graphtemp=.5ex\advance\graphtemp by 0.792in
    \rlap{\kern 3.300in\lower\graphtemp\hbox to 0pt{\hss $C_2$\hss}}%
    \graphtemp=.5ex\advance\graphtemp by 0.429in
    \rlap{\kern 0.891in\lower\graphtemp\hbox to 0pt{\hss $T_1$\hss}}%
    \graphtemp=.5ex\advance\graphtemp by 1.431in
    \rlap{\kern 1.794in\lower\graphtemp\hbox to 0pt{\hss $T_2$\hss}}%
    \hbox{\vrule depth1.584in width0pt height 0pt}%
    \kern 3.300in
  }%
}%
}

\vspace{-.5pc}
\caption{Decomposition of an induced subgraph of a $4$-colorable
graph.\label{fig:4col}}
\end{figure}

\begin{lemma}\label{goodcol}
Given a connected set $M$ in a $k$-colorable graph $G$, the induced subgraph
$G[M]$ has a good $k$-coloring.
\end{lemma}
\begin{proof}
Note that no edge of $G[M]$ can join two cycle-components, and no edge can
join two tree-components.  Hence the graph $H$ obtained from $G[M]$ by
contracting each tree-component and each cycle-component to a single vertex
is bipartite, with one part corresponding to the tree-components and the
other part to the cycle-components.  Furthermore, $H$ is acyclic, since
vertices in tree-components lie in no cycle in $G[M]$.

We produce a proper coloring of $G[M]$ with colors $1$ through $k$.  First
choose a vertex of $H$ and give the corresponding subgraph of $G[M]$ an optimal
proper coloring.  Next, for any unprocessed vertex of $H$ whose corresponding
subgraph $Q$ has a neighbor $v$ that is already colored, give $Q$ an optimal
coloring, using the color on $v$ as one of the colors if $Q$ has at least one
edge.  Do this until all of $G[M]$ has been colored.  Since $H$ is a tree, the
process succeeds and produces a proper $k$-coloring.  Furthermore, each set
consisting of the vertices of a tree-component and their neighbors in
cycle-components uses only two colors.
\end{proof}


To obtain the desired independent sets $\VEC{V'}1k$ such that
$\sum_i u(V'_i)\ge0$, we will start with sets $\VEC V1k$ forming a good
$k$-coloring of $G[M]$ and augment the sets by allowing some vertices with 
low degree in $G[M]$ to receive more than one color.

\section{Sparse $4$-Colorable Graphs}\label{sec:pla}

In this section we prove $\spo(G)\le \frac{8n+3m}5$ for every $4$-colorable
graph $G$ with $n$ vertices and $m$ edges, improving $\spo(G)<4n$ when $m<4n$.
Since $n$-vertex planar graphs are $4$-colorable and have at most $3n-6$ edges,
we obtain $\spo(G)\le 3.4n-3.6$ when $G$ is an $n$-vertex planar graph.

In order to apply the potential method, where
$\Phi(G) = \sum_{x \in V(G) \cup E(G)} \phi_G(x)$,
we specify $\phi_G$ when $G$ is a $4$-colorable graph.  In this section, let
\begin{equation*}
    \phi_G(x) = \left\{\def\arraystretch{1.4}\begin{array}{ll}
        \frac{3}{5} & \text{for } x \in E(G)\\
        1 + \frac{1}{5} \min\{\deg_G(x), 3\} & \text{for } x \in V(G) \\
        \end{array}\right.
\end{equation*}

Note that $\phi$ is monotone.  Since every vertex has potential at most
$\frac85$, proving $\spo(G)\le\Phi(G)$ implies the desired bound
$\spo(G)\le\frac{8n+3m}5$ in Theorem~\ref{thm:pla}(a).  Note that
$\spo(K_j)=\CH{j+1}2=\Phi(K_j)$ for $j\le4$.  In order to prove
$\spo(G)\le\Phi(G)$ by induction on $|V(G)|$, we have noted in
Section~\ref{sec:potentials} that it suffices to find, for each connected set
$M$, independent sets $\VEC {V'}14$ covering $M$ such that $\sum u(V'_i)\ge0$.
We obtain these independent sets from the sets $\VEC V14$ in a good
$4$-coloring of $G[M]$ by giving additional colors to some vertices of low
degree.

\begin{definition}\label{pathdef}
In a connected set $M$ in a $4$-colorable graph $G$,
let $P = \{v \in M\st \deg_{G[M]}(v) \leq 2 \}$.  Let $Q$ be the vertex set
of a component of $G[P]$.  Note that $G[Q]$ is a path or a cycle, and $Q$ is
contained in $T$ or is disjoint from $T$.

Let $R$ be a largest independent set in $G[Q]$, so $|R|=\CL{|Q|/2}$ when $G[Q]$
is a path and $|R|=\FL{|Q|/2}$ when $G[Q]$ is a cycle.  For $v\in R$,
add $v$ to all color classes that do not contain $v$ or any neighbor of $v$.
Since $R$ is independent in $G$, these additions can be made in any order.
Let $\VEC{V'}14$ be the resulting {\it augmented sets} containing $\VEC V14$,
respectively.
\end{definition}

To each color $i$, we have added only vertices with no neighbor in color $i$.
Hence the resulting sets $\VEC{V'}14$ are independent sets covering $M$.  It
remains only to prove $\SE i14 u(V'_i)\ge 0$.  The sum is the total utility
over each vertex in each set, grouped by the sets.  We can also group the
utility by vertices: let $u(v)=\SE i14 u_{V'_i}(v)$.  We will prove $u(v)\ge0$
when $v\in M-P$ and consider the vertices in $P$ grouped by their components in
$G[P]$.

For $x \in V(G)$, let $N_G(x)$ denote the set of neighbors of $x$ in $G$, and
let $c(x)$ denote the number of colors assigned to $x$.  Let
$s(x)=\sum_{y \in N_{G[M]}(x)} c(y)$.  A lemma greatly simplifies the subsequent case
analysis.

\begin{lemma}\label{lem:util}
If $\VEC{V'}14$ are the augmented sets covering a connected set $M$ in a
$4$-colorable graph $G$, and $x\in M$, then
\begin{equation*}
    u(x)\ge \left\{\def\arraystretch{1.4}\begin{array}{ll}
   \left(\FR45\deg_G(x)+1\right)c(x)+\FR15 s(x)-4 & \text{if } \deg_G(x)\le3\\
        4c(x)-4 & \text{if } \deg_G(x)\ge4
        \end{array}\right.
\end{equation*}
\end{lemma}

\begin{proof}
First suppose $\deg_G(x) \leq 3$.  If $x\in V'_i$, then
$u_{V'_i}(x) = \phi_G(x)+{\tsty\FR35}\deg_G(x) - 1 = {\tsty\FR45}\deg_G(x)$.
When $x\notin V'_i$, the difference between $\phi_G(x)$ and $\phi_{G-V'_i}(x)$
is $\FR15|N(x)\cap V'_i|$, since when $\deg_G(x)\le 3$ the potential decreases
by $\FR15$ for each lost neighbor.  Summing over the $4-c(x)$ values of $i$
such that $x\notin V'_i$, we obtain
\begin{equation*}
\sum_{i\st x \notin V'_i} u_{V'_i}(x)
= \sum_{i\st x \notin V'_i} \left({\tsty\FR15}\norm{N(x)\cap V'_i}-1\right)
= {\tsty\FR15} s(x) - \left(4 - c(x)\right).
\end{equation*}
Adding $\FR45\deg_G(x)$ for each of the $c(x)$ values of $i$ with
$x\in V'_i$ yields the desired value.

When $\deg_G(x)\ge 4$, we have $\phi_G(x)=\FR85$ and
$\phi_G(x)-\phi_{G-V'_i}(x)\ge0$.  Hence $u_{V'_i}(x)=\FR35(d_G(x)+1)$ if
$x\in V'_i$ and $u_{V'_i}(x)\ge -1$ if $x\notin V'_i$.  Again there are
$c(x)$ indices of the former type and $4-c(x)$ of the latter type, so
\begin{equation*}
u(x)\ge {\tsty\FR35}(d_G(x)+1)c(x)-[4-c(x)]
=({\tsty\FR35 d_G(x)+\FR85})c(x)-4\ge 4c(x)-4.
\end{equation*}

\vspace{-2.5pc}
\end{proof}

\medskip
As mentioned earlier, to complete the proof of Theorem~\ref{thm:pla}(a) it
suffices to prove the following lemma.

\begin{lemma}\label{lemma:nonneg}
If $\VEC{V'}14$ are the augmented sets covering a connected set $M$ in a
$4$-colorable graph $G$, then $\SE i14 u(V'_i)\ge0$.
\end{lemma}
\begin{proof}
We break the sum into its vertex contributions, with
$u(v)=\SE i14 u_{V'_i}(v)$ for $v\in V(G)$.
We first show $u(v)\ge0$ for $v\in V(G)-P$ and then group the vertices
of $P$ by components of $G[P]$.

If $v\in V(G)-M$, then $u_{V'_i}(v)=\phi_G(v)-\phi_{G-V'_i}(v)\ge0$ for each
$i$, which suffices.  If $v\in M-P$ with $d_G(v)\ge4$, then
Lemma~\ref{lem:util} yields $u(v)\ge4c(v)-4\ge0$, since $c(v)\ge1$.  If
$v\in M-P$ with $d_G(v)<4$, then $v\notin P$ requires $d_G(v)=d_{G[M]}(v)=3$.
Hence $c(v)\ge1$ and $s(v)\ge3$, which by Lemma~\ref{lem:util} yields
$u(v)\ge0$.

It remains to prove $\sum_{v\in Q} u(v)\ge0$ when $Q$ is the vertex set of 
a component of $G[P]$.  Recall that $G[Q]$ is a path or a cycle, and vertices
in a largest independent subset $R \subseteq Q$ have been assigned additional
colors.  Also $Q\esub T$ or $Q\cap T=\nul$.

\Case{$Q \subseteq T$}
If $d_{G[M]}(v)=0$, then $v$ is in all augmented sets and $u(v)\ge0$.  Hence we
may assume $\deg_G(v) \geq \deg_{G[M]}(v) \geq 1$.

If $v \in R$, then only one color is used on $N_{G[T]}(v)$, so $c(v)=3$ and
$s(v)\ge 1$.  By Lemma~\ref{lem:util}, $u(v) \geq \frac{8}{5}$.
On the other hand, if $v\in Q-R$, then $c(v) = 1$ and $s(v) \geq 3$.  Now
Lemma~\ref{lem:util} yields $u(v) \geq -\frac{8}{5}$.
Since $|R| \geq |Q-R|$, we thus obtain $\sum_{v\in Q}u(v)\ge0$.

\Case{$Q \cap T = \emptyset$}
For $v \in Q$ we have $\deg_G(v) \geq \deg_{G[M]}(v)=2$, since $v$ lies on
a cycle in $M$.

If $v \in R$, then $c(v), s(v) \geq 2$, and Lemma~\ref{lem:util} yields
$u(v) \geq \frac{8}{5}$.
If $v \notin R$, then $c(v) = 1$, and at least one neighbor of $v$ lies in
$R$; choose $w\in N(v)\cap R$.  Since $w\in P$, we have $d_{G[M]}(w)=2$, so $w$
receives at least one extra color.  Since $d{G[M]}(v)=2$, there is another
neighbor of $v$ with at least one color, so $s(v)\ge3$.  Since also
$d_G(v)\ge2$, Lemma~\ref{lem:util} now yields $u(v) \geq -\frac{4}{5}$.

Since always $\norm{R} \geq \FR13\norm{Q}$ (with equality when $G[Q]=K_3$),
we have $|R|\ge \FR12|Q-R|$, and hence $\sum_{v\in Q}u(v)\ge0$.
\end{proof}

\section{Outerplanar Graphs}

For the family of outerplanar graphs, we use a different potential function.
A {\it triangle} is a $3$-vertex complete graph.

\begin{definition}\label{pot}
For an outerplanar graph $G$, let
${\Phi(G) = \sum_{x \in V(G) \cup E(G)} \phi_G(x)}$, where
$$
\phi_G(x) =
\left\{
\begin{array}{rl}
  \smallskip
  \frac13 &\text{if $x$ is an edge of $G$,} \\
  \smallskip
  \frac53 & \text{if $x$ is a vertex in a triangle or having degree at least $3$
in $G$,} \\
  \smallskip
  \frac43 & \text{if $x$ is a vertex in no triangle and $x$ has degree
$1$ or $2$ in $G$,}\\
  \smallskip
  1 & \text{if $x$ is an isolated vertex in $G$.}
\end{array}
\right.
$$
Note that always $\phi_H(x)\le \phi_G(x)$ when $H\esub G$; thus
$\phi$ is monotone on this family.
\end{definition}

Figure \ref{fig:potentials} illustrates the contributions to potential for
vertices in an outerplanar graph; every edge contributes $\FR13$.  As
motivation for the definition, note that if $H$ is $K_n$ for $n\in\{1,2,3\}$,
then $\Phi(H)=\CH{n+1}2=\spo(H)$.  

\begin{figure}[h]
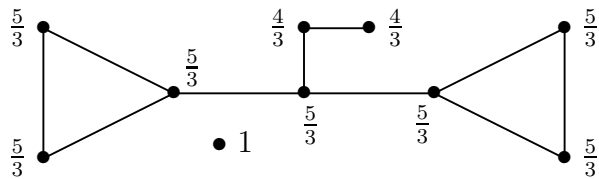

\gpic{
\expandafter\ifx\csname graph\endcsname\relax \csname newbox\endcsname\graph\fi
\expandafter\ifx\csname graphtemp\endcsname\relax \csname newdimen\endcsname\graphtemp\fi
\setbox\graph=\vtop{\vskip 0pt\hbox{%
    \graphtemp=.5ex\advance\graphtemp by 0.818in
    \rlap{\kern 0.136in\lower\graphtemp\hbox to 0pt{\hss $\bu$\hss}}%
    \graphtemp=.5ex\advance\graphtemp by 0.136in
    \rlap{\kern 0.136in\lower\graphtemp\hbox to 0pt{\hss $\bu$\hss}}%
    \graphtemp=.5ex\advance\graphtemp by 0.477in
    \rlap{\kern 0.818in\lower\graphtemp\hbox to 0pt{\hss $\bu$\hss}}%
    \graphtemp=.5ex\advance\graphtemp by 0.477in
    \rlap{\kern 1.500in\lower\graphtemp\hbox to 0pt{\hss $\bu$\hss}}%
    \graphtemp=.5ex\advance\graphtemp by 0.477in
    \rlap{\kern 2.182in\lower\graphtemp\hbox to 0pt{\hss $\bu$\hss}}%
    \graphtemp=.5ex\advance\graphtemp by 0.818in
    \rlap{\kern 2.864in\lower\graphtemp\hbox to 0pt{\hss $\bu$\hss}}%
    \graphtemp=.5ex\advance\graphtemp by 0.136in
    \rlap{\kern 2.864in\lower\graphtemp\hbox to 0pt{\hss $\bu$\hss}}%
    \graphtemp=.5ex\advance\graphtemp by 0.136in
    \rlap{\kern 1.500in\lower\graphtemp\hbox to 0pt{\hss $\bu$\hss}}%
    \graphtemp=.5ex\advance\graphtemp by 0.136in
    \rlap{\kern 1.841in\lower\graphtemp\hbox to 0pt{\hss $\bu$\hss}}%
    \graphtemp=.5ex\advance\graphtemp by 0.750in
    \rlap{\kern 1.057in\lower\graphtemp\hbox to 0pt{\hss $\bu$\hss}}%
    \special{pn 11}%
    \special{pa 818 477}%
    \special{pa 136 818}%
    \special{fp}%
    \special{pa 136 818}%
    \special{pa 136 136}%
    \special{fp}%
    \special{pa 136 136}%
    \special{pa 818 477}%
    \special{fp}%
    \special{pa 818 477}%
    \special{pa 2182 477}%
    \special{fp}%
    \special{pa 2182 477}%
    \special{pa 2864 818}%
    \special{fp}%
    \special{pa 2864 818}%
    \special{pa 2864 136}%
    \special{fp}%
    \special{pa 2864 136}%
    \special{pa 2182 477}%
    \special{fp}%
    \special{pa 1500 477}%
    \special{pa 1500 136}%
    \special{fp}%
    \special{pa 1500 136}%
    \special{pa 1841 136}%
    \special{fp}%
    \graphtemp=.5ex\advance\graphtemp by 0.818in
    \rlap{\kern 0.000in\lower\graphtemp\hbox to 0pt{\hss $\FR53$\hss}}%
    \graphtemp=.5ex\advance\graphtemp by 0.136in
    \rlap{\kern 0.000in\lower\graphtemp\hbox to 0pt{\hss $\FR53$\hss}}%
    \graphtemp=.5ex\advance\graphtemp by 0.347in
    \rlap{\kern 0.915in\lower\graphtemp\hbox to 0pt{\hss $\FR53$\hss}}%
    \graphtemp=.5ex\advance\graphtemp by 0.648in
    \rlap{\kern 1.534in\lower\graphtemp\hbox to 0pt{\hss $\FR53$\hss}}%
    \graphtemp=.5ex\advance\graphtemp by 0.642in
    \rlap{\kern 2.119in\lower\graphtemp\hbox to 0pt{\hss $\FR53$\hss}}%
    \graphtemp=.5ex\advance\graphtemp by 0.818in
    \rlap{\kern 3.000in\lower\graphtemp\hbox to 0pt{\hss $\FR53$\hss}}%
    \graphtemp=.5ex\advance\graphtemp by 0.136in
    \rlap{\kern 3.000in\lower\graphtemp\hbox to 0pt{\hss $\FR53$\hss}}%
    \graphtemp=.5ex\advance\graphtemp by 0.136in
    \rlap{\kern 1.364in\lower\graphtemp\hbox to 0pt{\hss $\FR43$\hss}}%
    \graphtemp=.5ex\advance\graphtemp by 0.136in
    \rlap{\kern 1.977in\lower\graphtemp\hbox to 0pt{\hss $\FR43$\hss}}%
    \graphtemp=.5ex\advance\graphtemp by 0.750in
    \rlap{\kern 1.193in\lower\graphtemp\hbox to 0pt{\hss $1$\hss}}%
    \hbox{\vrule depth0.955in width0pt height 0pt}%
    \kern 3.000in
  }%
}%
}

\vspace{-1pc}
\caption{Potentials of vertices.\label{fig:potentials}}
\end{figure}

A {\it maximal outerplanar graph} is an outerplanar graph that is not a
spanning subgraph of any other outerplanar graph.  For $n\ge3$, a maximal
outerplanar graph with $n$ vertices can be embedded in the plane so that
the boundary of the unbounded face is a spanning cycle and all bounded faces
are triangles.

\begin{lemma}\label{73n}
If $\spo(G)\le\Phi(G)$ whenever $G$ is a maximal outerplanar graph, then
$\spo(G)\le\FR73|V(G)|$ whenever $G$ is an outerplanar graph.
\end{lemma}

\begin{proof}
The desired bound holds by inspection when $\C{V(G)}\le2$.  By the monotonicity
of $\spo$, it suffices to prove $\spo(G)\le\FR73n$ when $n\ge3$ and $G$ is a
maximal outerplanar graph with $n$ vertices.  Such a graph $G$ has exactly
$2n-3$ edges and has every vertex in a triangle.  Hence each vertex has
potential $\FR53$, and
$$\spo(G)\le\Phi(G) = {\tsty\FR53}n + {\tsty\FR13}(2n-3) < {\tsty\FR73}n\text{.}$$

\vspace{-2pc}
\end{proof}

\bigskip
To prove Theorem~\ref{thm:outerplanar}, we will show that $\spo(G)\le\Phi(G)$
whenever $G$ is an induced subgraph of a maximal outerplanar graph.  The
relevant consequence of maximality here is that every vertex lying on a cycle
in $G$ in fact lies on a triangle.

The approach is as in Section~\ref{sec:potentials}.  With $\gamma=\FR13$ for
this potential function, and the connected set $M$ in $G$, we have the
definition of utility as in Definition~\ref{util}.
Lemma~\ref{prop:util} holds, and it suffices to prove $\sum_i u(V'_i)\ge 0$ for
independent sets $V'_1,V'_2,V'_3$ in $G$ covering $M$.
With {\it tree-components}, {\it cycle-components},
and {\it good $3$-coloring} defined as in Definition~\ref{components}, again
the proof of Lemma~\ref{goodcol} guaranteeing a good $3$-coloring 
$V_1,V_2,V_3$ is valid.  However, this time our method for obtaining
the augmented sets $V'_1,V'_2,V'_3$ is a bit different.

\begin{definition}\label{parts}
Given a connected set $M$ in an induced subgraph $G$ of a maximal outerplanar
graph $G$, with $M$ split into $T$ and $S$ as in Definition~\ref{components},
let $P=\{v\in T\st d_{G[M}(v)\le 2\}$.  Call each component of $G[P]$ a
{\it path-component}.  Let $N(T)$ be the set consisting of $T$ and all
neighbors in $G\sbrac{M}$ of vertices in $T$.
\end{definition}

The definition of $P$ here differs from Definition~\ref{pathdef} by restricting
$P$ to $T$.  For a connected set $M$ with $|M|\ge2$, each component of
$G[N(T)]$ is a tree with at least two vertices, on which the good $3$-coloring
guaranteed by Lemma~\ref{goodcol} uses two colors.  Since $G[M]$ is an induced
subgraph of a maximal outerplanar graph, each vertex of $S$ lies in a triangle
in $G[M]$.

\begin{definition}\label{augmop}
Given a connected set $M$ in an induced subgraph $G$ of a maximal outerplanar
graph, let $A,B,C$ be a good $3$-coloring of $G[M]$ as provided by
Lemma~\ref{goodcol}.  Let $Q$ be the vertex set of a path-component in $G[P]$,
having vertices $v_1, \ldots, v_k$ in order.  Let $Z$ be the color among
$\{A,B,C\}$ not initially used on the component of $N(T)$ containing $Q$.
If $\deg_G(v_1)\ge 2$, then add to the set of vertices with color $Z$ the
odd-indexed vertices $v_1, v_3,\ldots$.  If $\deg_G(v_1) = 1$, then instead add
the analogous set starting from the other end: $v_\ell,v_{\ell-2},\ldots$.
Do this independently for each path-component to produce
the {\it augmented sets} $A',B',C'$.
\end{definition}

Note that the augmented sets are independent, since in the good $3$-coloring
all neighbors in $G[M]$ of vertices in $Q$ receive colors other than $Z$.
The final lemma completes the proof of Theorem~\ref{thm:pla}(b).

\begin{lemma}\label{lem:potential_method}
If $G$ is an induced subgraph of a maximal outerplanar graph, then
$\slowcol{G} \le \Phi(G)$.
\end{lemma}

\begin{proof}
We use induction on $|V(G)|$.  When $\norm{V(G)}=1$, both $\spo(G)$ and
$\Phi(G)$ equal $1$.  For $\norm{V(G)}>1$, let $M$ be the initial set marked by
Lister, which we may assume is connected.  It suffices to prove
$u(A')+u(B')+u(C')\ge0$ for the augmented sets $A',B',C'$ in
Definition~\ref{augmop}.

If $\norm{M}=1$ with $M=\{v\}$, then setting $X=M$ satisfies
$u(X)=\Phi(G)-\Phi(G-X)-1\ge0$, since $\phi_G(v)\ge1$ and vertex potentials
cannot increase when taking subgraphs.  Therefore, we may assume $\C M\ge2$.
Since $G[M]$ is connected, this implies $\deg_G(v)\ge\deg_{G\sbrac{M}}(v)\ge1$
for $v\in M$, so all vertices of $M$ have potential at least $\FR43$ in $G$.

We prove the desired inequality by breaking the sum into its contributions from
individual vertices, as in Lemma~\ref{lemma:nonneg}.  For $v\in V(G)$, let
$u(v)=u_{A'}(v)+u_{B'}(v)+u_{C'}(v)$.  Since
$\sum_{v\in M} u(v)=u(A')+u(B')+u(C')$, the argument is completed by proving 
$u(v) \ge 0$ for $v \in V(G)-P$ and $\sum_{v \in Q} u(v) \ge 0$ for every path
component $G[Q]$.



Since always $\phi_G(v)\ge\phi_{G-X}(v)$ when $v\notin X$, for every
independent set $X\esub M$ we have $u_X(v)\ge0$ when $v\notin M$ and
$u_X(v)\ge-1$ when $v\in M$.  Hence we need only consider $v\in M$.

\bigskip
\noindent
{\bf CLAIM 1:} {\it $u(v)\ge0$ for $v\in M - P$.}
Vertex $v$ has exactly one color originally and after the augmentation; by
symmetry, we may assume $v \in A\esub A'$.  We consider cases depending on
$\deg_G(v)$.  If $\deg_G(v)\le1$, then $v$ cannot lie in a triange,
so $v\in T$, but then $v\in P$.  Hence for $v\in M-P$ we may assume
$\deg_G(v)\ge2$.

  \Cases

  \Case{$\deg_G(v) \ge 4$}
Since $\deg_G(v)\ge3$, we have $\phi_G(v)=\frac53$.  With $v\in A$ and
four incident edges, $u_{A'}(v)\ge\frac53 +4\cdot\frac13 - 1 = 2$.
We have noted $u_X(v)\ge -1$ for $X\in\{B',C'\}$, so $u(v)\ge0$.

  \Case{$\deg_G(v) = 3$}
First consider $v\in S$.  Choose $x, y \in M$ so that $\{v, x, y\}$ is a
triangle -- see Figure \ref{fig:deg_3_triangle}.

\begin{figure}[h]
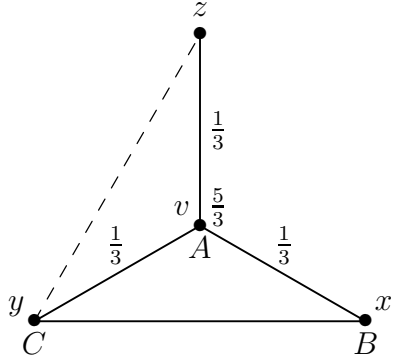

\gpic{
\expandafter\ifx\csname graph\endcsname\relax \csname newbox\endcsname\graph\fi
\expandafter\ifx\csname graphtemp\endcsname\relax \csname newdimen\endcsname\graphtemp\fi
\setbox\graph=\vtop{\vskip 0pt\hbox{%
    \graphtemp=.5ex\advance\graphtemp by 1.134in
    \rlap{\kern 1.000in\lower\graphtemp\hbox to 0pt{\hss $\bu$\hss}}%
    \graphtemp=.5ex\advance\graphtemp by 0.133in
    \rlap{\kern 1.000in\lower\graphtemp\hbox to 0pt{\hss $\bu$\hss}}%
    \graphtemp=.5ex\advance\graphtemp by 1.634in
    \rlap{\kern 1.867in\lower\graphtemp\hbox to 0pt{\hss $\bu$\hss}}%
    \graphtemp=.5ex\advance\graphtemp by 1.634in
    \rlap{\kern 0.133in\lower\graphtemp\hbox to 0pt{\hss $\bu$\hss}}%
    \special{pn 11}%
    \special{pa 1000 133}%
    \special{pa 1000 1134}%
    \special{fp}%
    \special{pa 1000 1134}%
    \special{pa 1867 1634}%
    \special{fp}%
    \special{pa 1867 1634}%
    \special{pa 133 1634}%
    \special{fp}%
    \special{pa 133 1634}%
    \special{pa 1000 1134}%
    \special{fp}%
    \special{pn 8}%
    \special{pa 133 1634}%
    \special{pa 1000 133}%
    \special{da 0.067}%
    \graphtemp=.5ex\advance\graphtemp by 0.000in
    \rlap{\kern 1.000in\lower\graphtemp\hbox to 0pt{\hss $z$\hss}}%
    \graphtemp=.5ex\advance\graphtemp by 1.540in
    \rlap{\kern 1.961in\lower\graphtemp\hbox to 0pt{\hss $x$\hss}}%
    \graphtemp=.5ex\advance\graphtemp by 1.540in
    \rlap{\kern 0.039in\lower\graphtemp\hbox to 0pt{\hss $y$\hss}}%
    \graphtemp=.5ex\advance\graphtemp by 1.268in
    \rlap{\kern 1.000in\lower\graphtemp\hbox to 0pt{\hss $A$\hss}}%
    \graphtemp=.5ex\advance\graphtemp by 1.768in
    \rlap{\kern 1.867in\lower\graphtemp\hbox to 0pt{\hss $B$\hss}}%
    \graphtemp=.5ex\advance\graphtemp by 1.768in
    \rlap{\kern 0.133in\lower\graphtemp\hbox to 0pt{\hss $C$\hss}}%
    \graphtemp=.5ex\advance\graphtemp by 1.040in
    \rlap{\kern 0.906in\lower\graphtemp\hbox to 0pt{\hss $v$\hss}}%
    \graphtemp=.5ex\advance\graphtemp by 1.040in
    \rlap{\kern 1.094in\lower\graphtemp\hbox to 0pt{\hss $\FR53$\hss}}%
    \graphtemp=.5ex\advance\graphtemp by 0.639in
    \rlap{\kern 1.094in\lower\graphtemp\hbox to 0pt{\hss $\FR13$\hss}}%
    \graphtemp=.5ex\advance\graphtemp by 1.240in
    \rlap{\kern 1.441in\lower\graphtemp\hbox to 0pt{\hss $\FR13$\hss}}%
    \graphtemp=.5ex\advance\graphtemp by 1.240in
    \rlap{\kern 0.559in\lower\graphtemp\hbox to 0pt{\hss $\FR13$\hss}}%
    \hbox{\vrule depth1.768in width0pt height 0pt}%
    \kern 2.000in
  }%
}%
}

\vspace{-.5pc}
\caption{Neighborhood of $v$ in Case 2 when $v\in S$ \label{fig:deg_3_triangle}}
\end{figure}

The vertices $v,x,y$ have different colors in the good $3$-coloring of $G[M]$,
so by symmetry we may assume $x\in B$ and $y\in C$.  Let $z$ be the neighbor of
$v$ outside $\{x,y\}$.  Since $G$ is outerplanar and hence cannot contain
$K_4$, vertex $z$ is not adjacent to both $x$ and $y$.  By symmetry, we may
assume $xz \notin E(G)$.  We have $u_{A'}(v)=\frac53+3\cdot\frac13-1 =\frac53$
and $u_{B'}(v)\ge-1$.  For $X=C$, note that after coloring $y$ the vertex $v$
will be in no triangle and have degree $2$, so $\phi_{G-X}(v)=\frac{4}{3}$.
Thus $u_{C'}(v)=\frac53-\frac43-1=-\frac23$, yielding $u(v)\ge 0$.

Now suppose $v\in T$.  All three neighbors of $v$ are in $M$, since otherwise
$v \in P$.  Since $v \in T$, the neighbors all have the same color; call
it $C$.  Again $u_{A'}(v) = \frac53 + 3 \cdot \frac13 - 1 = \frac{5}{3}$ and
$u_{B'}(v) \ge -1$.  In $G-C$, vertex $v$ is isolated, so
$u_{C'}(v)\ge\frac53-1-1 = -\frac13$; again $u(v) \ge 0$.

  \Case{$\deg_G(v) = 2$}
If $v \in T$, then $v \in P$, so we may assume $v \in S$.  Since
$d_G(v)=2$, exactly one triangle contains $v$, and its vertices have three
different colors.  Recall $v\in A$.  When $B'$ is deleted, $v$ is no longer in
a triangle, so $u_{B'}(v)\ge\frac53-\frac43-1=-\frac23$.  Similarly,
$u_{C'}(v)\ge -\frac23$.
Since $u_{A'}(v)\ge\frac53+2\cdot\frac13-1 = \frac{4}{3}$, we have $u(v) \ge 0$.

  \bigskip
\noindent
{\bf CLAIM 2:}
{\it $\SE i1\ell u(v_i)\ge0$ for a path-component with vertices
$\VEC v1\ell$ in order.}  Let $Q=\{\VEC v1\ell\}$.  By symmetry, let $A$ and
$B$ be the colors in the good $3$-coloring used on the tree-component
containing $Q$.  We consider two cases.

  \Cases

  \Case{$\ell=1$}
Let $Y\in\{A,B\}$ be the initial color on $v_1$, with $Z$ being the other
color in $\{A,B\}$.  Since $v_1$ is the first vertex of $Q$ from either end,
$v_1$ is added to $C'$ regardless of $d_G(v_1)$.  Since $G[M]$ is connected,
all vertices have potential at least $\frac43$, so
$u_{Y'}(v_1)\ge\FR43+\FR13-1=\FR23$.
Similarly $u_{C'}(v_1) \ge \frac{2}{3}$.  Since $u_{Z'}(v_1) \ge -1$, we have
$u(v_1) \ge \frac{1}{3} > 0$.

  \Case{$\ell \ge 2$}\label{cas:path_2}
Consider $v \in Q$.  Again let $Y\in\{A,B\}$ be the initial color on $v$ and
$Z\in\{A,B\}$ be the other initial color on $Q$.

First suppose $\deg_G(v) \ge 2$.  If $v$ is added to $C'$, then
$u_{Y'}(v)=u_{C'}(v)\ge\FR43+2\cdot\FR13-1=1$ and $u_{Z'}(v) \ge -1$, so
$u(v) \ge 1$.  If $v$ is not added to $C'$, then $u_{Y'}(v) \ge 1$,
$u_{Z'}(v) \ge -1$, and $u_{C'}(v) \ge -1$, so $u(v)\ge-1$.
Thus internal vertices of $G[Q]$ alternate bounds $u(v)\ge1$ and $u(v)\ge-1$.
\looseness-1

Degree $1$ in $G$ is possible when $v$ is an endpoint of $G[Q]$.  In that case
$u_{Y'}(v)\ge\frac43+\frac13-1=\frac23$.  Since deleting $Z'$ isolates $v$,
losing potential $\frac13$, we have $u_{Z'}(v)\ge-\frac23$.  If $v$ is added to
$C'$, then $u_{C'}(v)\ge\frac23$, and $u(v)\ge\frac23$.  If $v$ is not added to
$C'$, then its neighbor is added to $C'$, so $u_{Y'}(v)\ge\frac23$,
$u_{Z'}(v)\ge-\frac23$, and $u_{C'}(v)\ge -\frac23$, yielding
$u(v)\ge-\frac{2}{3}$.

If $\ell$ is odd, then $v_1$ and $v_\ell$ are both added to $C'$, and
$u(v_1)+u(v_\ell) \ge \frac{4}{3}$.  For the internal vertices,
$\SE i1{\ell-1} u(v_i) \ge -1$, so $\SE i1\ell u(v_i) \ge 0$.

If $\ell$ is even, then only one of $v_1$ and $v_\ell$ is added to $C'$.  This
endpoint contributes at least $\frac23$, and the other endpoint contributes at
least $-\frac23$.  Since the number of internal vertices is even, they also
contribute at least $0$, so $\SE i1\ell u(v_i) \ge 0$.
\end{proof}




\end{document}